\documentclass[11pt,a4paper]{article}

\usepackage{authblk}

\usepackage{mathrsfs}
\usepackage{amsfonts}
\usepackage{amsthm}
\usepackage{authblk}
\usepackage{cite}
\usepackage{amsmath,amssymb,amscd,amsthm}
\usepackage{graphics}
\usepackage{enumerate}
\input xy
\xyoption{all}

\numberwithin{equation}{section}

\usepackage{scalefnt}   

\usepackage{comment}
\usepackage{setspace}

\usepackage[paperwidth=210mm,
paperheight=297mm,
a4paper,
twoside,
top=0.8in, left=0.9in, right=0.9in, bottom=1in]{geometry}

\usepackage[active]{srcltx}

\theoremstyle{definition}
\newtheorem{lemma}{Lemma}[section]

\newtheorem{prop}[lemma]{Proposition}

\newtheorem{defi}[lemma]{Definition}
\newtheorem{thm}[lemma]{Theorem}

\usepackage[colorlinks=true, linkcolor=magenta, urlcolor=blue, citecolor=blue]{hyperref}

\newcommand{\C}{\mathbb{C}}

\newcommand{\Z}{\mathbb{Z}}

\def\a{\alpha}

\usepackage{authblk} 
\usepackage{hyperref}
\usepackage{float}
\restylefloat{figure}

\allowdisplaybreaks

\begin{document}
	\title{Generalized superderivations of the super Virasoro algebras}

    \author{Shun Liu\thanks{E-mail: 19020230157190@stu.xmu.edu.cn} \; and \, Dashu Xu\thanks{Corresponding author. E-mail: dox@mail.ustc.edu.cn}}

	\date{}
	\maketitle
	\begin{abstract}
	We explicitly determine the generalized superderivations of the Neveu-Schwarz and Ramond algebras.
	\bigskip
		
	\noindent {\em Key words: Generalized superderivation; Quasi superderivation; Super Virasoro algebra}
		
	\end{abstract}
	
	\section{Introduction}
     Super Virasoro algebras are significant algebraic objects in both mathematics and physics.
     They are super extensions of the Virasoro algebra, with their even part identical to the original Virasoro algebra, and play an important role in superconformal field theory and superstring theory.
     There exist two types of super-Virasoro algebras, namely the Neveu-Schwarz and Ramond algebras.
     These Lie superalgebras are generated by one bosonic current corresponding to the Virasoro algebra, and one fermionic current.
     In the field of Lie superalgebras, research on the super Virasoro algebras proceeds along two branches: structure and representation theory. 
     Abundant research achievements have been obtained concerning the super Virasoro algebras.
     In terms of representation theory, Verma modules have been thoroughly investigated in \cite{IK}, and complete classifications of simple Harish-Chandra modules and smooth modules have been established in \cite{S,CLL,CYZ,LPX}.
     Structural investigations have also yielded a wealth of substantial results. For instance, the classification of compatible left-symmetric superalgebra structures on super Virasoro algebras has been completed in \cite{KB}, while the local structural properties of super Virasoro algebras have been explored in \cite{W,DGL}.
     For additional academic findings, interested readers may consult the relevant references therein.
     
     The goal of this paper is to determine all generalized superderivations of the super Virasoro algebras.
     Generalized superderivations for Lie superalgebras were introduced in \cite{ZZ}, which is a super generalization of the concept for Lie algebras \cite{LL}.
     The authors of \cite{ZZ} studied generalized superderivations of finite-dimensional Lie superalgebras.
     The notion of generalized derivations has also been generalized to Hom-Lie algebras \cite{ZNC}, Lie conformal algebras \cite{FHS}, and other algebraic structures. 
     Interested readers may consult the citations of \cite{LL} and the review paper \cite{Kay}.
     The remainder of this paper is organized as follows.
     Section \ref{sec2} is devoted to recalling some basic notions necessary for the present paper.
     In Section \ref{sec3}, we completely determine the generalized superderivations of the Neveu-Schwarz and Ramond algebras.
     Our computations build upon existing investigations concerning generalized derivations of the Virasoro algebra \cite{KKS}.
     
     Throughout this paper, we use $\C$, $\Z$,  $\Z^*$ and $\Z_2$ to denote the complex number field, the set of integers, the set of non-zero integers, and the ring of integers modulo 2, respectively.
     
    \section{Preliminaries}\label{sec2}
    Let $\mathfrak{g}=\mathfrak{g}_{\overline{0}}\oplus\mathfrak{g}_{\overline{1}}$ be a Lie superalgebra.
    For $\theta\in\Z_2$, recall that $$\mathrm{End}(\mathfrak{g})_{\theta}=\left\{f\in\mathrm{End}(\mathfrak{g})\mid f(\mathfrak{g}_{\mu})\subseteq\mathfrak{g}_{\mu+\theta},\,\forall\mu\in\Z_2\right\}.$$
    
    \noindent A homogeneous superderivation of degree $\theta\in\Z_2$ is a map $f\in\mathrm{End}(\mathfrak{g})_{\theta}$ such that
    $$[f(x),y]+(-1)^{\alpha\theta}[x,f(y)]=f\left([x,y]\right)$$
    for all $x\in\mathfrak{g}_{\alpha}$ and $y\in\mathfrak{g}$.
    Denote the space of homogeneous superderivations of degree $\theta\in\Z_2$ by $\mathfrak{Der}(\mathfrak{g})_{\theta}$.
    The space $\mathfrak{Der}(\mathfrak{g})=\mathfrak{Der}(\mathfrak{g})_{\overline{0}}\oplus\mathfrak{Der}(\mathfrak{g})_{\overline{1}}$ is called the superderivation algebra of $\mathfrak{g}$.

    \begin{defi}
    A map $f\in\mathrm{End}(\mathfrak{g})_{\theta}$ is called a generalized superderivation of degree $\theta\in\Z_2$, provided that there exist $f',f''\in\mathrm{End}(\mathfrak{g})_{\theta}$ such that 
    $$[f(x),y]+(-1)^{\alpha\theta}[x,f'(y)]=f''([x,y])$$
    for $x\in\mathfrak{g}_{\alpha}$ and $y\in\mathfrak{g}$.
    If $f=f'$, then $f$ is called a quasi superderivation of degree $\theta\in\Z_2$.
    If $f=-f'$ and $f''=0$, then $f$ is called a super quasi-centroid map of degree $\theta\in\Z_2$.
    \end{defi}
    Denote by $\mathfrak{GDer}(\mathfrak{g})_{\theta}$, $\mathfrak{QDer}(\mathfrak{g})_{\theta}$, and $\mathfrak{QC}(\mathfrak{g})_{\theta}$ the spaces of generalized superderivations, quasi superderivations, and super quasi-centroid maps of degree $\theta\in\mathbb{Z}_2$, respectively.
    Define
    \begin{align*}
    	\mathfrak{GDer}(\mathfrak{g})&:=\mathfrak{GDer}(\mathfrak{g})_{\overline{0}}\oplus\mathfrak{GDer}(\mathfrak{g})_{\overline{1}},\\
    	\mathfrak{QDer}(\mathfrak{g})&:=\mathfrak{QDer}(\mathfrak{g})_{\overline{0}}\oplus\mathfrak{QDer}(\mathfrak{g})_{\overline{1}},\\
    	\mathfrak{QC}(\mathfrak{g})&:=\mathfrak{QC}(\mathfrak{g})_{\overline{0}}\oplus\mathfrak{QC}(\mathfrak{g})_{\overline{1}}.
    \end{align*}
    Due to \cite[Proposition 2.4]{ZZ}, we have 
    \begin{equation}\label{decom}
    \mathfrak{GDer}(\mathfrak{g})=\mathfrak{QDer}(\mathfrak{g})+\mathfrak{QC}(\mathfrak{g}).
	\end{equation}
    \textbf{The above statements reduce to the ordinary Lie algebra case \cite{LL} when $\mathfrak{g}_{\overline{1}}=0$}.
     
    Next, let us introduce the super Virasoro algebras.
    \begin{defi}
    For $\epsilon=0$ or $\dfrac{1}{2}$, the super Virasoro algebra $\mathfrak{SVir}[\epsilon]$ is a Lie superalgebra with a basis $\left\{L_m,G_r,c\mid m\in\mathbb{Z},\,r\in \epsilon+\mathbb{Z}\right\}$. Its non-trivial Lie superbrackets are given by
    \begin{align}
    	\label{eq1}
    	[L_m,L_n]&=(m-n)L_{m+n}+\dfrac{1}{12}(m^3-m)\delta_{m+n,0}c,\\
    	\label{eq2}
    	[L_m,G_r]&=\left(\dfrac{m}{2}-r\right)G_{m+r},\\
    	[G_r,G_s]&=2L_{r+s}+\dfrac{1}{12}(4r^2-1)\delta_{r+s,0}c.
    \end{align}
    The even part $\mathfrak{SVir}[\epsilon]_{\overline{0}}$ is spanned by $\left\{L_m,c\mid m\in\mathbb{Z}\right\}$, whereas the odd part $\mathfrak{SVir}[\epsilon]_{\overline{1}}$ is spanned by $\left\{G_r\mid r\in\epsilon+\mathbb{Z}\right\}$. Specifically, $\mathfrak{SVir}[0]$ is called the Ramond algebra, while $\mathfrak{SVir}\left[\dfrac{1}{2}\right]$ is called the Neveu–Schwarz algebra.
    \end{defi}
    In fact, the even part of $\mathfrak{SVir}[\epsilon]$ is precisely the Virasoro algebra, which we shall simply denote $\mathfrak{SVir}[\epsilon]_{\overline{0}}$ by $\mathfrak{Vir}$. It is known from \cite[Proposition 2.2]{ZM} that every derivation of $\mathfrak{Vir}$ is inner.
    
    Let $\mathrm{Ann}(\mathfrak{Vir})=\left\{g\in\mathrm{End}(\mathfrak{Vir})\mid g(\mathfrak{Vir})\subseteq\C c\right\}$.
    The following theorem then holds.
    \begin{thm}\cite[Theorem 16]{KKS}\label{QVir}
    	$\mathfrak{QDer}(\mathfrak{Vir})=\mathfrak{Der}(\mathfrak{Vir})\oplus\C \mathrm{id}_{\mathfrak{Vir}}\oplus\mathrm{Ann}(\mathfrak{Vir})$.
    \end{thm}

    \section{Generalized superderivations of the super Virasoro algebras}\label{sec3}
    For $\epsilon=0$ or $\dfrac{1}{2}$, define
    $$\mathrm{Ann}\left(\mathfrak{SVir}[\epsilon]\right):=\left\{f\in\mathrm{End}\left(\mathfrak{SVir}[\epsilon]\right)\mid f\left(\mathfrak{SVir}[\epsilon]\right)\subseteq\C c \right\}.$$
    For $\theta\in\Z_2$, denote by $\mathrm{Ann}\left(\mathfrak{SVir}[\epsilon]\right)_{\theta}=\mathrm{Ann}\left(\mathfrak{SVir}[\epsilon]\right)\cap\mathrm{End}\left(\mathfrak{SVir}[\epsilon]\right)_\theta$.
    In the following two subsections, we will determine the generalized superderivations of the Ramond and the Neveu–Schwarz algebras, respectively.
    \subsection{Generalized superderivations of the Ramond algebra}

     According to \eqref{decom}, we determine the quasi superderivations of the Ramond algebra first.
    \begin{thm}\label{QSD-R}
        $\mathfrak{QDer}(\mathfrak{SVir}[0])=\mathfrak{Der}(\mathfrak{SVir}[0])\oplus\C \mathrm{id}_{\mathfrak{SVir}[0]}\oplus\mathrm{Ann}(\mathfrak{SVir}[0])$. 
    \begin{proof}        
        Step 1.
        $\mathfrak{QDer}(\mathfrak{SVir}[0])_{\overline{0}}=\mathfrak{Der}(\mathfrak{SVir}[0])_{\overline{0}}\oplus\C \mathrm{id}_{\mathfrak{SVir}[0]}\oplus\mathrm{Ann}(\mathfrak{SVir}[0])_{\overline{0}}.$

        Let $f$ be a quasi superderivation of degree $\overline{0}\in\Z_2$ with the corresponding even map $f'\in\mathrm{End}\left(\mathfrak{SVir}[0]\right)_{\overline{0}}$.
        Then, 
        $f\upharpoonright_{\mathfrak{Vir}}:\mathfrak{Vir}\to\mathfrak{Vir}$ is a quasi derivation of $\mathfrak{Vir}$.
        According to Theorem \ref{QVir},
        we may assume $f\upharpoonright_{\mathfrak{Vir}}\in\mathrm{Ann}\left(\mathfrak{Vir}\right)$.
        
        Suppose
        \begin{align*}
            f\left ( G_m \right ) =\sum_{k \in \Z}Y(m,k) G_k,
            \quad
            f'\left ( G_m \right ) =\sum _{k \in \Z}Y'(m,k)G_k,\quad\text{and}\quad f\left ( \mathfrak{Vir}\right) \subseteq\C c.
       \end{align*}
    From the equation
    \[
    \left [ f\left ( G_m \right ) ,G_n \right ] +\left [ G_m, f\left ( G_n \right )  \right ] =f'\left ( [ G_m,G_n  ]  \right )=0,
    \]
    we obtain
    \begin{equation}\label{GR-11}
        Y\left ( m,k-n \right )+Y \left ( n,k-m \right )=0.
    \end{equation}
    Setting $m=n$ in Equation \eqref{GR-11} yields 
    \[
    Y(m,n)=0 \quad \text{for }   m, n\in \Z.
    \]
    This shows that $f \in \mathrm{Ann}(\mathfrak{SVir}[0])_{\overline{0}}$.
    Since $\mathfrak{SVir}[0]$ is a perfect Lie algebra,
    we deduce that $f'=0$.
    Consequently, we have
    \[
    \mathfrak{QDer}(\mathfrak{SVir}[0])_{\overline{0}}=\mathfrak{Der}(\mathfrak{SVir}[0])_{\overline{0}}\oplus\C \mathrm{id}_{\mathfrak{SVir}[0]}\oplus\mathrm{Ann}(\mathfrak{SVir}[0])_{\overline{0}}.
    \]
   Step 2.
        $\mathfrak{QDer}(\mathfrak{SVir}[0])_{\overline{1}}=\mathfrak{Der}(\mathfrak{SVir}[0])_{\overline{1}}\oplus\mathrm{Ann}(\mathfrak{SVir}[0])_{\overline{1}}.$

        Let $h$ be a quasi superderivation of degree $\overline{1}\in\Z_2$ with the corresponding odd map $h'\in\mathrm{End}\left(\mathfrak{SVir}[0]\right)_{\overline{1}}$. 
        For $x \in \mathfrak{SVir}[0]_{\a}$ and $y \in \mathfrak{SVir}[0]$,
        by definition we have 
        \begin{equation}\label{GR-2}
        \left [ h\left ( x \right ),y  \right ] +\left ( -1 \right )^{\alpha}  \left [x, h\left ( y \right )\right ]=h'\left ( \left [ x,y \right ]  \right ). 
        \end{equation}
      
        Assume that 
        \begin{align*}
            h\left ( L_m \right ) & =\sum_{k \in \Z}B\left ( m,k \right ) G_k, \qquad \qquad\qquad h'\left ( L_m \right )  =\sum_{k \in \Z}B'\left ( m,k \right ) G_k, \\
            h\left ( G_m \right ) &=\sum_{k \in \Z}X\left ( m,k \right ) L_k+Z\left ( m \right ) c, \qquad h'\left ( G_m \right ) =\sum_{k \in \Z}X'\left ( m,k \right ) L_k+Z'\left ( m \right ) c,\\
            h\left ( c \right ) &=0, \qquad \qquad\qquad\qquad\qquad \quad \qquad h'\left ( c \right ) =\sum_{k \in \Z}F'\left ( k \right ) G_k. 
        \end{align*}
        Define a superderivation of degree $\overline{1}\in\Z_2$ as follows:
        \[
        \varphi =\dfrac{1}{2}\sum _{k \in \Z}X\left ( 0,k \right )\mathrm{ad}\left ( G_k \right ).
        \]
        Replacing $h$ by $h-\varphi$,
        it follows that one can assume $ X\left ( 0,k \right )=0$ for all $k \in \Z$.
        
        Substituting $\left ( x,y \right )  =\left ( L_m,G_n \right ) $ into Equation \eqref{GR-2} and comparing the coefficients of $L_k$ and $c$,
        we obtain
        \begin{align}
            2B\left ( m,k-n \right ) +\left ( 2m-k \right ) X\left ( n,k-m \right ) &=\left ( \dfrac{m}{2}-n  \right )X'\left ( m+n,k \right ), \label{GR-21} \\ 
            \left ( 4n^2-1 \right ) B\left ( m,-n \right ) +\left ( m^3-m \right ) X\left ( n,-m \right ) &=\left ( 6m-12n \right )Z'\left ( m+n \right ). \label{GR-22} 
        \end{align}
       Setting $n=0$ in Equation \eqref{GR-21} yields
       \[
       4B\left ( m,k \right ) =mX'\left ( m,k \right ), 
       \]
        which shows $B\left ( 0,k \right ) =0$ for $k \in \Z$.
        Substituting $m=0$ in Equation \eqref{GR-21} gives
        \[
        kX\left ( n,k \right ) =nX'\left ( n,k \right ).
        \]
        From Equation \eqref{GR-21},
        we obtain
        \[
        2\left ( m+n \right ) B\left ( m,k-n \right ) +\left ( m+n \right ) \left ( 2m-k \right )  X\left ( n,k-m \right )=\left ( \dfrac{m}{2}-n  \right )\left ( m+n \right ) X'\left ( m+n,k \right ).
        \]
        From the foregoing relationships between $X$, $B$, and $X'$, we have
    \begin{equation}\label{Gr-23}
    (m+n)(k-n)X(m,k-n)+2(m+n)(2m-k)X(n,k-m)=(m-2n)kX(m+n,k).
    \end{equation}
    Setting $k=0$ in Equation \eqref{Gr-23} yields
    \begin{equation}\label{Gr-24}
    -n(m+n)X(m,-n)+4m(m+n)X(n,-m)=0.
    \end{equation}

    Substituting $m=0$ in Equation \eqref{GR-22} gives $Z'(m)=0$ for $m\in\mathbb{Z}^*$.
    If $m+n\neq0$ in Equation  \eqref{GR-22}, we obtain
    \[
    (4n^2-1)B(m,-n)+(m^3-m)X(n,-m)=0.
    \]
    Applying the relationships among $X$, $B$, and $X'$ to the above equation yields
    \begin{equation}\label{Gr-25}
    -n(4n^2-1)X(m,-n)+4(m^3-m)X(n,-m)=0.
    \end{equation}

    Combining Equations \eqref{Gr-24} and \eqref{Gr-25}, we have
    \[
    m(m-2n)(m+2n)X(n,-m)=0.
    \]
    Hence
    $X(m,n)=0$ for $ n\neq0, n\neq m,\text{and } n\neq\pm2m$.
    
    (1) Setting $k=m=n\neq0$ in Equation \eqref{Gr-23} yields  
    \[
    4m^2X(m,0)=-m^2X(2m,m).
    \]  
    Since $m\neq0$ (which implies $m\neq2m$ and $m\neq\pm4m$), we obtain $X(2m,m)=0$. Consequently,  
    \[
    X(m,0)=0\quad\text{for } m\in\mathbb{Z}^*.
    \]
    (2) Setting $m\neq0$, $n=m$, and $k=-m$ in Equation \eqref{Gr-23} gives  
    \[
    8m^2X(m,-2m)=m^2X(2m,-m).
    \]  
    Because $m\neq0$ (so that $-m\neq2m$ and $-m\neq\pm4m$), we have $X(2m,-m)=0$. Hence  
    \[
    X(m,-2m)=0\quad\text{for } m\in\mathbb{Z}^*.
    \]
    (3) Setting $m\neq0$, $n=m$, and $k=3m$ in Equation \eqref{Gr-23} leads to  
    \[
    4m^2X(m,2m)=3m^2X(2m,3m).
    \]  
    Since $m\neq0$ (which ensures $3m\neq2m$ and $3m\neq\pm4m$), it follows that $X(2m,3m)=0$. Thus  
    \[
    X(m,2m)=0\quad\text{for } m\in\mathbb{Z}^*.
    \]
    (4) Substituting $k=m+n$ into Equation \eqref{Gr-23} yields, for $m+n\neq0$,
    \begin{equation}\label{Gr-26}
    mX(m,m)+2(m-n)X(n,n)=(m-2n)X(m+n,m+n).
    \end{equation}
    Evaluating Equation \eqref{Gr-26} at specific pairs $(m,n)$ gives the following relations:
    \begin{align*}
    (m,n)=(1,1)\quad\Rightarrow \quad& X(2,2) = -X(1,1),\\
    (m,n)=(1,2)\quad\Rightarrow \quad& X(3,3) = -X(1,1),\\
    (m,n)=(3,1)\quad\Rightarrow \quad& X(4,4) = X(1,1),\\
    (m,n)=(4,1)\quad\Rightarrow \quad& X(5,5) = 5X(1,1),\\
    (m,n)=(2,3)\quad\Rightarrow \quad& X(5,5) = 0.
    \end{align*}
    The last two equalities together imply $X(1,1)=0$, and consequently
    \[
    X(n,n)=0\quad\text{for } n=1,2,3,4,5.
    \]
    Now assume $n\ge5$ and that $X(n,n)=0$. Setting $m=1$ in Equation \eqref{Gr-26} (which is valid since $n+1\neq0$) gives
    \[
    X(1,1)+2(1-n)X(n,n)=(1-2n)X(n+1,n+1).
    \]
    Hence $X(n+1,n+1)=0$. By induction, $X(n,n)=0$ for all $n\ge1$.
    For negative indices, 
    take $m<0$ and set $n=1-2m$ in Equation \eqref{Gr-26}. 
    This yields $X(m,m)=0$.
    We conclude that
    \[
    X(n,n)=0\quad\text{for all } n\in\mathbb{Z}^*,
    \]
   and thus
    \[
    X(m,n)=0\quad\text{for all } m, n\in\mathbb{Z}.
    \]

    Since $4B(m,k)=mX'(m,k)$ and $kX(n,k)=nX'(n,k)$, we have  
    \[
    4B(m,k)=kX(m,k)=0.
    \]  
    Hence $h \in \mathrm{Ann}(\mathfrak{SVir}[0])_{\overline{1}}$.
    Consequently,  
    we obtain
    \[
    \mathfrak{QDer}(\mathfrak{SVir}[0])_{\overline{1}}=\mathfrak{Der}(\mathfrak{SVir}[0])_{\overline{1}} \oplus \mathrm{Ann}(\mathfrak{SVir}[0])_{\overline{1}}.
    \]

    Combining Steps 1 and 2, we obtain the direct sum decomposition  
    \[
    \mathfrak{QDer}(\mathfrak{SVir}[0])=\mathfrak{Der}(\mathfrak{SVir}[0])\oplus\mathbb{C}\,\mathrm{id}_{\mathfrak{SVir}[0]}\oplus\mathrm{Ann}(\mathfrak{SVir}[0]). \qedhere
    \]  
    \end{proof}
    \end{thm}
     Our next step is to compute the super quasi-centroid maps of the Ramond algebra.
     
    \begin{prop}\label{SQC-R}
     $ \mathfrak{QC}(\mathfrak{SVir}[0])=\mathbb{C}\mathrm{id}_{\mathfrak{SVir}[0]}\oplus\mathrm{Ann}(\mathfrak{SVir}[0]).$   
     \begin{proof}
         Let $f$ be a homogeneous super quasi-centroid map of $\mathfrak{SVir}[0]$.
         If $f \in\mathrm{End}\left(\mathfrak{SVir}\left[0\right]\right)_{\overline{0}}$,
         then for $x,y\in\mathfrak{SVir}[0]$, 
         \begin{equation}\label{SQC-R1}
         \left [ f\left ( x \right ), y  \right ] -\left [ x, f\left ( y \right )  \right ]=0.   
         \end{equation}
         Assume further that 
         \begin{equation}\label{SQC-R1-1}
         \begin{aligned}
           f\left ( L_m \right) &=\sum_{k \in \Z}A \left( m,k \right)L_k+Z\left ( m\right )c, \\
          f\left ( G_n \right) &=\sum_{k \in \Z}Y\left ( n,k \right )G_k, \\
          f\left ( c \right) & = Ec.
         \end{aligned}  
         \end{equation}

        Substituting $\left ( x,y \right ) =\left ( L_m, L_n \right )$ into Equation \eqref{SQC-R1},
        we have 
     \begin{align}
         \left ( k-2n \right )A\left ( m,k-n \right )&=\left ( 2m-k \right )A\left ( n,k-m \right ),   \label{SQC-R2} \\
         \left ( n-n^3 \right ) A\left ( m,-n \right ) &=\left ( m^3-m \right ) A\left ( n,-m \right ). \label{SQC-R3} 
     \end{align}
     Setting $k=2m$ in Equation \eqref{SQC-R2} gives
     \[
      A\left ( m,n \right ) =0 \quad \text{for } m \ne n.
     \]
    Substituting $\left ( k,n \right ) =\left ( m+1,1 \right ) $ into Equation \eqref{SQC-R2} yields
    \[
    A\left ( m,m \right )=A\left ( 1,1 \right )  \quad \text{for } m \ne 1.
    \]
    Hence,
    we have 
    \[
    A\left ( m,n \right )=\delta_{mn} A\left ( 1,1 \right ) \quad \text{for all } m, n \in \Z.
    \]

      Substituting $\left ( x,y \right ) =\left ( L_m, G_n \right )$ into Equation \eqref{SQC-R1},
      we obtain
      \[
      \sum_{k \in \Z} \left ( k-3n \right ) A\left ( m,k-n \right ) G_k=\sum_{k \in \Z} \left ( 3m-2k\right ) Y\left ( n,k-m \right ) G_k.  
      \]
     It follows that
     \begin{equation}\label{SQC-R5}
         \left ( k-3n\right ) A\left ( m,k-n \right )=\left ( 3m-2k\right ) Y\left ( n,k-m \right ). 
     \end{equation}
    Setting $\left ( k,m \right ) =\left ( 3n+1,2n+1 \right ) $ in Equation 
    \eqref{SQC-R5} gives
    \[
    Y\left ( n,n\right ) =A\left ( 2n+1,2n+1 \right )=A\left ( 1,1 \right ). 
    \]
    Taking $\left ( k,m\right ) =\left ( 3n+3l+1, 2n+2l+1 \right ) $ with $l \ne0$ in Equation 
    \eqref{SQC-R5},
    we obtain 
    \[
    Y\left ( n,n+l\right ) =\left ( 3l+1 \right ) A\left ( 2n+2l+1,2n+3l+1 \right )=0.
    \]
    Hence,
    \[
    Y\left ( m,n\right ) =\delta_{mn} A\left ( 1,1\right ) \quad \text{for } m, n \in \Z. 
    \]

    Now,
    Equation \eqref{SQC-R1-1} reduces to 
    \begin{align*}
          f\left ( L_m \right) &=A \left( 1,1 \right)L_m+Z\left ( m\right )c, \\
          f\left ( G_n \right) &=A\left ( 1,1 \right )G_n, \\
          f\left ( c \right) & = Ec,
    \end{align*} 
    which implies
    \[
     \mathfrak{QC}(\mathfrak{SVir}[0])_{\overline{0}} =\mathbb{C}\mathrm{id}_{\mathfrak{SVir}[0]}\oplus\mathrm{Ann}(\mathfrak{SVir}[0])_{\overline{0}}.
    \]

    Now, suppose $h \in\mathrm{End}\left(\mathfrak{SVir}\left[0\right]\right)_{\overline{1}}$.
    Then, for $x\in\mathfrak{SVir}\left[0\right]_\alpha$ and $y\in\mathfrak{SVir}\left[0\right]$, we have 
    \begin{equation}\label{SQC-T1}
    \left [ h\left ( x \right ), y  \right ] =\left ( -1 \right )^{\alpha}  \left [ x, h\left ( y \right )  \right ].   
    \end{equation}
    Assume that 
    \begin{equation}\label{SQC-T1-1}
    \begin{aligned}
     h\left ( L_m \right) &=\sum_{k \in \Z}B \left( m,k \right)G_k, \\
          h\left ( G_n \right) &=\sum_{k \in \Z}X\left ( n,k \right )L_k+Z\left ( n \right )c, \\
          h\left ( c \right) & = 0.
         \end{aligned}  
         \end{equation}
    
    Substituting $\left ( x,y \right ) =\left ( L_m, L_n \right )$ into Equation \eqref{SQC-T1} gives
    \begin{equation}\label{SQC-T1-2}
    \left ( 2k-3n \right ) B\left ( m,k-n \right ) =\left ( 3m-2k \right ) B\left ( n,k-m \right ).
     \end{equation}
    Setting $m=n$ in the Equation \eqref{SQC-T1-2} yields
    \[
     \left ( 3m-2k \right ) B\left ( m,k-m \right )=0.
    \]
    It follows that 
    $B\left ( m,n \right )=0$  for $m\ne n$.  
    Taking $ \left ( k,m \right )= \left( 2n,n \right )$ in Equation \eqref{SQC-T1-2} yields
    \[
    B\left ( m, m \right )=0  \quad \text{for } m \ne 0. 
    \]
     Setting $ \left ( m, n, k \right )= \left( 1, 0, 1 \right )$ in Equation \eqref{SQC-T1-2} gives
     $
     B\left ( 0, 0\right )=2B\left ( 1, 1\right )=0$.
     Hence,
     we obtain
     \[
    B\left ( m, n \right )=0  \quad \text{for } m, n \in \Z. 
    \]

    Substituting  $\left ( x,y \right ) =\left ( L_m, G_n \right )$ into Equation \eqref{SQC-T1} yields 
    $\left ( 2m-k \right ) X \left ( n, k-m\right )=0$.
    Taking $ k=2m+1$ gives
     \[
     X\left ( m, n \right )=0 \quad \text{for all } m, n \in \Z.
     \]
    Thus $h \in \mathrm{Ann}(\mathfrak{SVir}[0])_{\overline{1}}$.
    Consequently, we have
    $
    \mathfrak{QC}(\mathfrak{SVir}[0])=\mathbb{C}\mathrm{id}_{\mathfrak{SVir}[0]}\oplus\mathrm{Ann}(\mathfrak{SVir}[0])$.
    \end{proof}
    \end{prop}
    
    Combining Theorem \ref{QSD-R} and the above proposition, we conclude that $$\mathfrak{GDer}(\mathfrak{SVir}[0])=\mathfrak{QDer}(\mathfrak{SVir}[0])=\mathfrak{Der}(\mathfrak{SVir}[0])\oplus\C \mathrm{id}_{\mathfrak{SVir}[0]}\oplus\mathrm{Ann}(\mathfrak{SVir}[0]).$$
    Thus, we have determined all generalized superderivations of the Ramond algebra.
    
    \subsection{Generalized superderivations of the Neveu–Schwarz  algebra}
    Similar to the case of Ramond algebra, we first determine the quasi superderivations of the Neveu–Schwarz algebra.
    \begin{thm}\label{QSD-NS}
        $\mathfrak{QDer}\left(\mathfrak{SVir}\left[\dfrac{1}{2} \right]\right)=\mathfrak{Der}\left(\mathfrak{SVir}\left[\dfrac{1}{2} \right]\right)\oplus\C \mathrm{id}_{\mathfrak{SVir}\left[\frac{1}{2} \right]}\oplus\mathrm{Ann}\left(\mathfrak{SVir}\left[\dfrac{1}{2} \right]\right).$ 
        \begin{proof}
         Step 1.
        $\mathfrak{QDer}\left(\mathfrak{SVir}\left[\dfrac{1}{2} \right]\right)_{\overline{0}}=\mathfrak{Der}\left(\mathfrak{SVir}\left[\dfrac{1}{2} \right]\right)_{\overline{0}}\oplus\C \mathrm{id}_{\mathfrak{SVir}\left[\frac{1}{2} \right]}\oplus\mathrm{Ann}\left(\mathfrak{SVir}\left[\dfrac{1}{2} \right]\right)_{\overline{0}}.$

        Let $f$ be a quasi superderivation of degree $\overline{0}\in\Z_2$ with the corresponding even map $f'\in\mathrm{End}\left(\mathfrak{SVir}\left[\dfrac{1}{2} \right]\right)_{\overline{0}}$.
        Then, 
        $f\upharpoonright_{\mathfrak{Vir}}:\mathfrak{Vir}\to\mathfrak{Vir}$ is a quasi derivation of $\mathfrak{Vir}$.
        By Theorem \ref{QVir},
        we may assume $f\upharpoonright_{\mathfrak{Vir}}\in\mathrm{Ann}\left(\mathfrak{Vir}\right)$.
        
        Suppose $f\left ( \mathfrak{Vir}\right)  \subseteq\C c$ and
        \begin{align*}
            f\left ( G_{m+\frac{1}{2} } \right ) &=\sum_{k \in \Z}Y(m+\dfrac{1}{2},k+\dfrac{1}{2}) G_{k+\frac{1}{2}},\\
            f'\left ( G_{m+\frac{1}{2} } \right ) &=\sum _{k \in \Z}Y'(m+\dfrac{1}{2},k+\dfrac{1}{2})G_{k+\frac{1}{2}}.
       \end{align*}
    From the equation
    \[
    \left [ f\left ( G_{ m+\frac{1}{2}} \right ) ,G_{n+\frac{1}{2} } \right ] +\left [ G_{ m+\frac{1}{2}}, f\left ( G_{n+\frac{1}{2} } \right )  \right ] =f'\left ( \left [  G_{ m+\frac{1}{2}},G_{n+\frac{1}{2}}  \right ]\right ),
    \]
    we deduce that
    \begin{equation}\label{QSD-NS-1-1}
       Y\left ( m+\dfrac{1}{2} ,k-n +\dfrac{1}{2}\right )+Y \left ( n+\dfrac{1}{2},k-m+\dfrac{1}{2} \right )=0.
    \end{equation} 
    Setting $m=n$ in Equation \eqref{QSD-NS-1-1} gives
    \[
    Y\left ( m+\dfrac{1}{2}, n+\dfrac{1}{2}\right )=0  \quad\text{for } m, n \in \Z.
    \]
    This shows that $f \in \mathrm{Ann}\left(\mathfrak{SVir}\left[\dfrac{1}{2} \right]\right)_{\overline{0}}$.
    Since $\mathfrak{SVir}\left[\dfrac{1}{2}\right]$ is a perfect Lie algebra,
    we deduce that $f'=0$.
    Consequently, we obtain
    \[
    \mathfrak{QDer}\left(\mathfrak{SVir}\left[\dfrac{1}{2} \right]\right)_{\overline{0}}=\mathfrak{Der}\left(\mathfrak{SVir}\left[\dfrac{1}{2} \right]\right)_{\overline{0}}\oplus\C \mathrm{id}_{\mathfrak{SVir}\left[\frac{1}{2} \right]}\oplus\mathrm{Ann}\left(\mathfrak{SVir}\left[\dfrac{1}{2} \right]\right)_{\overline{0}}.
    \]

    Step 2.
        $\mathfrak{QDer}\left(\mathfrak{SVir}\left[\dfrac{1}{2} \right]\right)_{\overline{1}}=\mathfrak{Der}\left(\mathfrak{SVir}\left[\dfrac{1}{2} \right]\right)_{\overline{1}}\oplus\mathrm{Ann}\left(\mathfrak{SVir}\left[\dfrac{1}{2} \right]\right)_{\overline{1}}.$

        Let $h$ be a quasi superderivation of degree $\overline{1}\in\Z_2$ with the corresponding odd map $h'\in\mathrm{End}\left(\mathfrak{SVir}\left[\dfrac{1}{2}\right]\right)_{\overline{1}}$. 
        For $x \in \mathfrak{SVir}\left[\dfrac{1}{2}\right]_{\a}$ and $y \in \mathfrak{SVir}\left[\dfrac{1}{2}\right]$,
        by definition we have 
        \begin{equation}\label{GNS-2}
        \left [ h\left ( x \right ),y  \right ] +\left ( -1 \right )^{\alpha}  \left [x, h\left ( y \right )\right ]=h'\left ( \left [ x,y \right ]  \right ). 
        \end{equation}

        Assume that 
        \begin{align*}
    h(L_m) &= \sum_{k\in\Z} B\left(m,k+\dfrac{1}{2}\right) G_{ k+\frac{1}{2}}, &
    h'(L_m) &= \sum_{k\in\Z} B'\left(m,k+\dfrac{1}{2}\right) G_{k+\frac12}, \\[10pt]
     h \left(G_{m+\frac12}\right) &= \sum_{k\in\Z} X\left(m+\dfrac{1}{2},k\right) L_k + Z\left(m+\dfrac{1}{2}\right) c, &
     h'\left(G_{m+\frac12}\right) &=  \sum_{k\in\Z} X'\left(m+\dfrac{1}{2},k\right) L_k + Z'\left(m+\dfrac{1}{2}\right) c, \\[10pt]
    h(c) &= 0, &
    h'(c) &= \sum_{k\in\Z} F'\left(k+\dfrac{1}{2}\right) G_{k+\frac12}.
    \end{align*}
        Define a superderivation of degree $\overline{1}\in\Z_2$ as follows:
        \[
        \varphi =\sum _{k \in \Z}\dfrac{2B\left ( 0,k+\dfrac{1}{2}  \right )}{2k+1}\mathrm{ad}\left ( G_{k+\frac{1}{2}}\right ).
        \]
        Replacing $h$ by $h-\varphi$, it follows that one can assume $ B\left ( 0,k+\dfrac{1}{2} \right )=0$ for all $k \in \Z$.

        Substituting $\left ( x,y \right )  =\left ( L_m,G_{n+\frac{1}{2}}\right ) $ into Equation \eqref{GNS-2} and comparing the coefficients of $L_k$ and $c$,
        we obtain
        \begin{align}
           4B\left ( m,k-n-\dfrac{1}{2}  \right ) +\left ( 4m-2k \right ) X\left ( n+\dfrac{1}{2} ,k-m \right ) &=\left ( m-2n-1  \right )X'\left ( m+n+\dfrac{1}{2} ,k \right ), \label{GNS21}  \\ 
            \left ( 4n^2+4n \right ) B\left ( m,-n-\dfrac{1}{2} \right ) +\left ( m^3-m \right ) X\left ( n+\dfrac{1}{2} ,-m \right ) &=\left ( 6m-12n-6 \right )Z'\left ( m+n+\dfrac{1}{2} \right ). \label{GNS22} 
        \end{align}
    Setting $m=0$ in Equation \eqref{GNS21} yields
    \begin{equation}\label{GNS22-1}
       kX\left ( n+\dfrac{1}{2},k  \right ) =\left ( n+\dfrac{1}{2} \right )X'\left ( n+\dfrac{1}{2},k \right ).   
    \end{equation}    
    Taking $k=0$ in the Equation \eqref{GNS22-1} gives
    \[
        \left ( 2n+1 \right )X'\left ( n+\dfrac{1}{2}, 0 \right ) =0,  
    \]
    hence $X'\left ( n+\dfrac{1}{2}, 0 \right ) =0 $  for all $n \in \Z$.
    Substituting $k=0$ into Equation \eqref{GNS21} gives
    \begin{equation}\label{GNS23-1} 
        B\left ( m,-n-\dfrac{1}{2}\right )+mX\left (n+\dfrac{1}{2},-m \right)=0.  
    \end{equation}
    Setting $m=0$ in Equation \eqref{GNS22} yields
    \[
    \left ( 2n+1 \right ) Z'\left (n+\dfrac{1}{2} \right)=0,  
    \]
    which implies that $Z'\left (n+\dfrac{1}{2} \right)=0$ for all $ n \in \Z$.
    Combining Equations \eqref{GNS22} and  \eqref{GNS23-1},
    we obtain 
 \[
m\left ( 2n+1 +m \right )\left ( 2n+1 -m \right ) X\left ( n+\dfrac{1}{2},m  \right )=0.
 \]
    Thus
    \[
    X\left ( n+\dfrac{1}{2},m  \right )=0 \quad \text{for } m\ne0, \text{  } m \ne \pm\left ( 2n+1 \right ).
    \]
    In particular,
    $X\left ( n+\dfrac{1}{2},m  \right )=0$ whenever $m\in2\Z^*$.

    Now, combine Equations \eqref{GNS21}, \eqref{GNS22-1}, and \eqref{GNS23-1}, we have
    \begin{align}
     &-2m\left ( m+n+\dfrac{1}{2}  \right ) X\left ( n-k+\dfrac{1}{2}, -m  \right )+ \left ( 2m-k \right )\left ( m+n+\dfrac{1}{2}  \right )X\left ( n+\dfrac{1}{2} ,k-m \right ) \notag \\[10pt]
     =&\;\dfrac{k\left ( m-2n-1 \right )}{2}X\left ( m+n+\dfrac{1}{2},k  \right ).  \label{GNS3}  
    \end{align}   
    (1) Setting $k=m=2$ in Equation \eqref{GNS3} gives 
    $\left ( 2n+5  \right ) X\left ( n+\dfrac{1}{2},0  \right )=0$,   
    hence
    \[
    X\left( m+\dfrac{1}{2},0 \right)=0\quad\text{for } m\in\mathbb{Z}.
    \]
    (2) Take $(m,k)=(7-2n,4)$.
    Since $k$ is  even,
    we have $X\left ( m+n+\dfrac{1}{2},k  \right ) =0$.
    Then
    \[
    X\left ( n+\dfrac{1}{2} ,k-m \right )=X\left ( n+\dfrac{1}{2},2n-3\right )=0
    \]
    because $2n-3\ne \pm(2n+1)$.
    Substituting  $(m,k)=(7-2n,4)$ into Equation \eqref{GNS3} yields
    \[
    \left ( 2n-7 \right ) \left ( 2n-15  \right )X\left ( n-\dfrac{7}{2}, 2n-7   \right ) =0. 
    \]
     Consequently,  
    \[
    X \left(m+\dfrac{1}{2},2m+1 \right)=0\quad\text{for } m\in\mathbb{Z}.
    \]
    (3) Take $(m,k)=(2n-7,4)$.
    We still have $X\left ( m+n+\dfrac{1}{2},k  \right ) =0$.
    Now
    $X\left ( n+\dfrac{1}{2} ,k-m \right )=X\left ( n+\dfrac{1}{2}, 11-2n\right )=0$
    because $2n-11\ne -2n-1$.
     Substituting  $(m,k)=(2n-7,4)$ into Equation \eqref{GNS3} yields
    \[
    \left ( 2n-7 \right ) \left ( 6n-13  \right )X\left ( n-\dfrac{7}{2}, 7-2n    \right ) =0,
    \]
    hence
    \[
    X \left(m+\dfrac{1}{2}, -2m-1   \right)=0\quad\text{for } m\in\mathbb{Z}.
    \]
    
    From the above statements, we conclude that
    \[
    X\left ( m+\dfrac{1}{2}, n  \right )=0 \quad \text{for } m, n \in \Z.
    \]
    Then Equation \eqref{GNS23-1} suggests that
    \[
    B\left ( m, n+\dfrac{1}{2}  \right )=0 \quad \text{for } m, n \in \Z.
    \]
    Therefore $h \in \mathrm{Ann}\left(\mathfrak{SVir}\left[\dfrac{1}{2} \right]\right)_{\overline{1}}$,
    and consequently,  
    we obtain
    \[
    \mathfrak{QDer}\left(\mathfrak{SVir}\left[\dfrac{1}{2} \right]\right)_{\overline{1}}=\mathfrak{Der}\left(\mathfrak{SVir}\left[\dfrac{1}{2} \right]\right)_{\overline{1}} \oplus \mathrm{Ann}\left(\mathfrak{SVir}\left[\dfrac{1}{2} \right]\right)_{\overline{1}}.
    \]  
    Combining Steps 1 and 2, we obtain the direct sum decomposition  
    \[
    \mathfrak{QDer}\left(\mathfrak{SVir}\left[\dfrac{1}{2} \right]\right)=\mathfrak{Der}\left(\mathfrak{SVir}\left[\dfrac{1}{2} \right]\right)\oplus\mathbb{C}\,\mathrm{id}_{\mathfrak{SVir}\left[\frac{1}{2}\right]}\oplus\mathrm{Ann}\left(\mathfrak{SVir}\left[\dfrac{1}{2} \right]\right). \qedhere
    \]  
      \end{proof}
    \end{thm}

    Similar to Proposition \ref{SQC-R}, the following proposition is true. 
    \begin{prop}
     $ \mathfrak{QC}\left(\mathfrak{SVir}\left[\dfrac{1}{2} \right]\right)=\mathbb{C}\mathrm{id}_{\mathfrak{SVir}\left[\frac{1}{2}\right]}\oplus\mathrm{Ann}\left(\mathfrak{SVir}\left[\dfrac{1}{2} \right]\right).$
    \end{prop}
    \noindent Hence, we conclude that $$\mathfrak{GDer}\left(\mathfrak{SVir}\left[\dfrac{1}{2}\right]\right)=\mathfrak{QDer}\left(\mathfrak{SVir}\left[\dfrac{1}{2}\right]\right)=\mathfrak{Der}\left(\mathfrak{SVir}\left[\dfrac{1}{2} \right]\right)\oplus\C \mathrm{id}_{\mathfrak{SVir}\left[\frac{1}{2} \right]}\oplus\mathrm{Ann}\left(\mathfrak{SVir}\left[\dfrac{1}{2} \right]\right).$$

\end{document}